\documentclass[11pt] {amsart}
\usepackage{amsthm, mathtools, amsmath,amscd,graphicx,amssymb, pdfsync,
amstext, epsfig,verbatim}
\newtheorem{thm}{Theorem}
                                             
\newtheorem{lem}{Lemma}
\newtheorem{prop}{Proposition}
\newtheorem{cor}{Corollary}[section]
{\theoremstyle{remark}
\newtheorem {Rem}{Remark}}
{\theoremstyle{definition}
}

\def\fchar{\mathrm{char}}

\newcommand{\mF}{\mathbb F}

                                                   \newcommand{\CC}{\mathcal C}

\newcommand{\eps}{\varepsilon}

\newcommand{\la}{\lambda}

\newcommand{\Eta}{\mathrm{H}}

\newcommand{\beq}{\begin{equation}}
\newcommand{\eeq}{\end{equation}}

\newcommand{\QQ}{\mathbb Q}

\begin{document}
\begin{abstract}
Let $\mathcal{C}$ be a hyperelliptic curve of genus $g>1$ over an algebraically closed field $K$ of characteristic zero and $\mathcal{O}$ one of the $(2g+2)$ Weierstrass points in $\mathcal{C}(K)$.
Let $J$ be the jacobian of $\mathcal{C}$, which is a $g$-dimensional abelian variety over $K$.
Let us consider the canonical embedding of $C$ into $J$ that sends $\mathcal{O}$ to the zero of the group law on $J$. This embedding allows us to identify $\mathcal{C}(K)$ with a certain subset of the commutative group $J(K)$.

A special case of the famous theorem of Raynaud (Manin--Mumford conjecture) asserts that the set of torsion points in $\mathcal{C}(K)$ is finite.
It is well known that the points of order 2 in $\mathcal{C}(K)$ are exactly the ``remaining'' $(2g+1)$ Weierstrass points. One of the authors \cite{ZarhinIzv19} proved that  there are no torsion points of order $n$ in $\mathcal{C}(K)$ if $3\le n\le 2g$. So, it is natural to study torsion points of order $2g+1$  (notice that the number of such points in $\mathcal{C}(K)$ is always even).
Recently, the authors proved that there are infinitely many (for a given $g$)  mutually nonisomorphic pairs $(\mathcal{C},\mathcal{O})$ such that
$\mathcal{C}(K)$ contains at least four points of order $2g+1$.

In the present paper we prove that (for a given $g$) there are at most finitely many (up to a isomorphism) pairs $(\mathcal{C},\mathcal{O})$ such that
$\mathcal{C}(K)$ contains at least six points of order $2g+1$.

\end{abstract}

 \title[Torsion points of small order ]{Torsion points of small order on hyperelliptic curves}
\author {Boris M. Bekker}
\address{St. Petersburg State University, Department of Mathematics and Mechanics,   Universitetsky prospect, 28, Peterhof, St. Petersburg, 198504, Russia}
\email{ bekker.boris@gmail.com}
\author {Yuri G. Zarhin}
\address{Pennsylvania State University, Department of Mathematics, University Park, PA 16802, USA}
\email{zarhin@math.psu.edu}
\thanks{The second named author (Y.Z.) was partially supported by Simons Foundation Collaboration grant   \# 585711.}


\maketitle
\section{Introduction}

Let $K$ be an algebraically closed field of characteristic $\ne 2$ and  $\mathcal{C}$  a  smooth irreducible projective curve   over $K$
 of genus $g>1$. Let us fix a point $\mathcal{O}\in C(K)$.
Let $J$ be the jacobian of $\mathcal{C}$, which is a $g$-dimensional abelian variety over $K$.  There is a canonical closed
regular
 embedding $\mathcal{C} \hookrightarrow J$ that assigns to a point $P \in \mathcal{C}(K)$ the linear equivalence class of the divisor $(P) -(\mathcal{O})$ and sends $\mathcal{O}$ to the zero of the group law on $J$. We identify $\mathcal{C}$ with its image in $J$.  Then $\mathcal{C}(K)$ becomes a subset of the commutative group $J(K)$. A  theorem of Raynaud (Manin--Mumford conjecture) asserts that if $\mathrm{char}(K)=0$, then for any embedding $\mathcal{C} \hookrightarrow J$ the set of torsion points in $\CC(K)$ is finite \cite{Raynaud}. The importance of  determining explicitly the finite set occurring   in Raynaud's theorem
for specific curves  was
pointed out
by R. Coleman, G. Anderson, K. Ribet,  M. Kim, P. Tsermias, and other authors (see \cite{Ribet}, \cite{Tsermias}, \cite{Anderson1,Anderson2}).

In this paper we continue our study of torsion points of small order on odd degree hyperelliptic curves that was started in \cite{ZarhinIzv19,BekkerZarhin20}.
 In what follows  $\mathcal{C}$ is a genus $g>1$ hyperelliptic curve over $K$, and $\mathcal{O} \in \mathcal{C}(K)$ one of the $(2g+2)$ Weierstrass points on $C$.  The pair $(\mathcal{C},\mathcal{O})$ is called a {\sl marked genus $g$ odd degree hyperelliptic curve over} $K$. Let $\iota: \mathcal{C} \to \mathcal{C}$ be the hyperelliptic involution. Its fixed points are exactly the Weierstrass points, including $\mathcal{O}$.
 It is well known that $\iota$ acts on $\mathcal{C}(K)$ as multiplication by $-1$ in $J(K)$, i.e.,
$$\iota(P)=-P \in J(K) \ \text{for all}\ P \in \mathcal{C}(K) \subset J(K).$$

We are interested in points of $\mathcal{C}(K)$ that are torsion points of small order with respect to the group law on $J$.
Clearly, $\mathcal{C}(K)$ contains the only point of order $1$, namely $\infty$. It is well known that points of order $2$ on $\mathcal{C}$ are precisely all the remaining $(2g+1)$ Weierstrass points of order $2$.   It was proven in \cite{ZarhinIzv19} that if $d$ is a positive integer such that $3 \le d \le 2g$, then $\mathcal{C}(K)$ does {\sl not} contain torsion points of order $d$. (The case $g=2$ and $d=3,4$ was done earlier in \cite{BoxallGrant00}.)

 If $d>2$ is an integer and $P\in \mathcal{C}(K)$  is a torsion point of order $d$, then  $\iota(P) \in \mathcal{C}$  is also a torsion point
of the same order $d$ and $\iota(P)\neq P$. Hence  the number of torsion points of order $d$ in  $\mathcal{C}(K)$ is {\sl even}. We write $\mathcal{C}_{d}^{*}(K)$ for the set of torsion points of order $d$ in $\mathcal{C}(K)$, whose cardinality $\#(\mathcal{C}_{d}^{*}(K))$ is a nonnegative even integer.

The following results are proven in \cite{BekkerZarhin20}. (The case $g=2$  was treated earlier in \cite{BGL01}. See also
 \cite[Example 5.3 on pp. 349--350]{DJZ}.)

\begin{itemize}
\item[(i)]
There always exists $\mathcal{C}$ such that $\mathcal{C}(K)$ contains (at least) two points of order $2g+1$.
\item[(ii)]
If $2g+1$ is a power of $\fchar(K)$,  then $\#(\mathcal{C}_{2g+1}^{*}(K))\le 2$. (One may prove that there are infinitely many mutually non-isomorphic   $\mathcal{C}$'s
with $\#(\mathcal{C}_{2g+1}^{*}(K))= 2$.)

\item[(iii)]
If $2g+1$ is {\sl not} a power of $\fchar(K)$,  then there are infinitely many mutually non-isomorphic $\mathcal{C}$ with  $\#(\mathcal{C}_{2g+1}^{*}(K))\ge 4$.
\end{itemize}

On the other hand, J. Boxall \cite[Prop. 1.3]{Boxall20} proved that
$$\#(\mathcal{C}_{2g+1}^{*}(K))\le 8g^2.$$

The aim of this paper is to prove the following assertions.

\begin{thm}
\label{finiteness3}
Let $g>1$ be an integer and $K$ be an algebraically closed field.
Suppose that one of the following conditions holds.

\begin{itemize}
\item[(a)]
$\fchar(K)=0$.
\item[(b)]
$q=2g+1$ is a prime, $p:=\fchar(K)$ is an odd prime that is a primitive root modulo $q$.
\end{itemize}

  There exists a  finite (may be empty) set $S(g,K)$ of marked genus $g$ hyperelliptic curves over $K$ that enjoys the following properties.

\begin{enumerate}
\item[(1)]
$\#(S(g,K)) \le 9 (4g-1) \binom{2g}{g}^2$.

\item[(2)]
If $(\mathcal{C},\infty)$ is a marked genus $g$ hyperelliptic curve over $K$ such that
$\#(\mathcal{C}_{2g+1}^{*}(K))> 4$, then $(\mathcal{C},\infty)$ is  isomorphic to one of the marked curves in $S(g,K)$.
\end{enumerate}
\end{thm}

\begin{Rem}
In the case of $g=2$ the  finiteness assertion of Theorem \ref{finiteness3} was proven in \cite[Prop. 4.4]{BGL01} in all characteristics.
The paper \cite{BGL01} also contains explicit examples of genus 2 curves with 6 points of order 5. It would be interesting to find similar examples for $g>2$ (or to prove that they do not exist).
\end{Rem}

\begin{thm}
\label{primeFields}
Let $g>1$ be an integer and  $K$ be an algebraically closed field.
Suppose that one of the following conditions holds.

\begin{itemize}
\item[(a)]
$\fchar(K)=0$.
\item[(b)]
$q=2g+1$ is a prime, $p:=\fchar(K)$ is an odd prime that is a primitive root modulo $q$.
\end{itemize}

Let $(\mathcal{C},\mathcal{O})$ be a genus $g$ marked hyperelliptic curve over $K$.
Suppose that   $\#(\mathcal{C}_{2g+1}^{*}(K))> 4$. Then

\begin{itemize}
\item[(i)]
 If $\fchar(K)=0$, then $\mathcal{C}$ is defined over a number field.
\item[(ii)]
If $\fchar(K)=p>0$, then $\mathcal{C}$ is defined over a finite field.
\end{itemize}
\end{thm}

\begin{Rem}
Let $q=2g+1$ be a prime. It follows from Dirichlet's theorem about primes in arithmetic progression that the density  of primes $p$  enjoying property (b) of Theorem \ref{primeFields} is positive.

\end{Rem}

The paper is organized as follows. Section \ref{sec2} contains auxiliary results about certain polynomials related to  cyclotomic polynomials.
In Section \ref{sec3} we remind results from \cite{BekkerZarhin20} about hyperelliptic curves with two pairs of torsion points of order $2g+1$.
We use them in Section \ref{sec4} in order to  prove Theorems \ref{finiteness3} and \ref{primeFields}.


\section{Preliminaries}
\label{sec2}
Let $K_0$ be the algebraic closure of the prime subfield of $K$ in $K$. So, $K_0$ is an algebraic closure of the field $\QQ$ of rational numbers if
$\fchar(K)=0$ and $K_0$ is an algebraic closure of the prime finite field $\mF_p$ if $\fchar(K)=p>0$. In other words, $K_0$ is the smallest algebraically closed subfield of $K$.
Let $g>1$ be an integer such that $\fchar(K)\nmid (2g+1)$.  Let us consider the $2g$-element set
$$M(2g+1):=\{\eps \in K, \eps^{2g+1}=1, \eps\ne 1 \} \subset K^{*}.$$
We write
 $$\eta(\eps)= \frac1{\eps-1}  \ \forall \eps \in M(2g+1).$$
 Notice that the degree $2g$ polynomial
$$(x+1)^{2g+1}-x^{2g+1}=(2g+1)\prod_{\eps\in M(2g+1)}(x-\eta(\eps))\in K_0[x]\subset K[x]$$
has no multiple roots. Similarly,  if $b \in K^{*}$, then the degree $2g$ polynomial
$$x^{2g+1}-(x-b)^{2g+1}=(2g+1)b \prod_{\eps\in M(2g+1)}(x+b\eta(\eps))\in K_0[x]\subset K[x]$$
 has no multiple roots as well.  We have in the polynomial ring $K[x]$
 $$\sum_{i=0}^{2g} x^i=\frac{x^{2g+1}-1}{x-1}=\prod_{\eps\in M(2g+1)}(x-\eps).$$
 This implies that
 \begin{equation}
 \label{Vieta}
\sum_{\eps\in M(2g+1)}\eps=-1.
 \end{equation}

 If $I,J$  are subsets of $M(2g+1)$, then we write
 $$\Eta_I(x):=\prod_{\eps\in I}(x-\eta(\eps))\in K_0[x]\subset K[x]$$
and
 $$\Psi_{J,b}(x):= \prod_{\eps\in J}(x+b\eta(\eps))\in K_0[x]\subset K[x].$$
 We write ${\complement I}$
  for the complement of $I$ and ${\complement J}$ for the complement of $J$  in $M(2g+1)$.  Clearly, $\Eta_I(x)$ and $\Psi_{J,b}(x)$ are monic polynomials of degree $\#(I)$ and $\#(J)$ respectively. In addition,
 $$(x+1)^{2g+1}-x^{2g+1}=(2g+1)\Eta_I(x) \Eta_{\complement I}(x),$$
 $$x^{2g+1}-(x-b)^{2g+1}=(2g+1)b \Psi_{J,b}(x) \Psi_{{\complement J},b}(x).$$
 This implies easily the following assertion.

 \begin{prop}
 \label{basic}
 Let $b \in K^{*}$ and
 $$u_1(x),u_2(x), w_1(x), w_2(x) \in K[x]$$
 be degree $g$ polynomials such  that
 $$u_1(x) u_2(x)=(x+1)^{2g+1}-x^{2g+1}, \ w_1(x)w_2(x)=x^{2g+1}-(x-b)^{2g+1}.$$
 Then there exist $g$-element subsets $I,J \subset M(2g+1)$ and $\mu,\nu \in K^{*}$ such  that
 $$u_1(x)=\mu\Eta_I(x), u_2(x)=\frac{2g+1}{\mu}\Eta_{\complement I}(x);$$ $$w_1(x)=\nu\Psi_{J,b}(x), w_2(x)=\frac{(2g+1)b}{\nu} \Psi_{{\complement J},b}(x).$$
 \end{prop}

 \begin{cor}
 \label{old}
 Let $b$  be an element of $K$ such that $b \ne 0, -1$ and
 $v_1(x)$, $v_2(x)$, $v_3(x) \in K[x]$ be polynomials such that all $\deg (v_i) \le g$ and
 $$x^{2g+1}+v_1^2(x)=(x-b)^{2g+1}+v_2^2(x)=(x+1)^{2g+1}+v_3^2(x).$$
 Then there exist $g$-element subsets $I,J \subset M(2g+1)$ and $\mu,\nu \in K^{*}$ such  that
 $$v_1(x)=\frac{\mu}2\Eta_I(x)+\frac{2g+1}{2\mu}\Eta_{\complement I}(x)=
\frac{\nu}2\Psi_{J,b}(x)-\frac{(2g+1)b}{2\nu}\Psi_{\complement J,b}(x).$$
 \end{cor}

 \begin{proof}
 We have
 $$\begin{aligned}(x+1)^{2g+1}-x^{2g+1}=v_1^2(x)-v_3^2(x)\\=\left(v_1(x)+v_3(x)\right)\left(v_1(x)-v_3(x)\right)=u_1(x)u_2(x),\end{aligned}$$
 where
 $$u_1(x):=v_1(x)+v_3(x), \ u_2(x):=v_1(x)-v_3(x),$$
 and
 $$\begin{aligned}x^{2g+1}-(x-b)^{2g+1}=v_2^2(x)-v_1^2(x)\\=\left(v_1(x)+v_2(x)\right)\left(v_1(x)-v_2(x)\right)=w_1(x)w_2(x),\end{aligned}$$
 where
 $$w_1(x):=v_1(x)+v_2(x), \ w_2(x):=v_2(x)-v_1(x).$$
 Clearly,
 $$\deg(u_1)\le g,\ \deg(u_2) \le g,\  \deg(u_1)+\deg(u_2)=\deg(u_1 u_2)=2g,$$
 $$\deg(w_1)\le g,\  \deg(w_2) \le g, \ \deg(w_1)+\deg(w_2)=\deg(w_1 w_2)=2g.$$
 This implies that
 $\deg(u_i)=g, \deg(w_i)=g$ for $i=1,2$.

 Now the desired result follows from Proposition \ref{basic} combined with obvious formulas
 $$v_1(x)=\frac{u_1(x)+u_2(x)}{2}=\frac{w_1(x)-w_2(x)}{2}.$$

 \end{proof}

\section{Hyperelliptic curves}
\label{sec3}

Throughout this section we assume that  $\fchar(K)$ does {\sl not} divide  $2(2g+1)$.
Let $f(x) \in K[x]$ be a degree $2g+1$ polynomial without repeated roots. One may attach to $f(x)$ the  marked genus $g$ odd degree hyperelliptic curve $\mathcal{C}_f$ over $K$ that is
the smooth projective model of
$\mathcal{C}_f: y^2=f(x)$ of the smooth plane affine curve
$y^2=f(x)$  with precisely one  ``infinite'' point, which is denoted  $\infty$.  We have
$$\mathcal{C}_f(K)=\{(a,b)\in K^2\mid f(a)=b\} \sqcup \{\infty\}.$$
The point $\infty$ is a Weierstrass point of $\mathcal{C}_f$, so the pair $(\mathcal{C}_f,\infty)$
is  a marked genus $g$ odd degree hyperelliptic curve over $K$.

It is well known  \cite{Lock} that  if $(\mathcal{C},\mathcal{O})$  is  a marked genus $g$ odd degree hyperelliptic curve over $K$, then there exists a degree $2g+1$ polynomial $f(x)\in K[x]$ without repeated roots such that $(\mathcal{C},\mathcal{O})$  is biregularly isomorphic to $(\mathcal{C}_f,\infty)$ over $K$.

Let $(\mathcal{C},\mathcal{O})$ be a marked genus $g$ odd degree hyperelliptic curve over $K$
such that  $\mathcal{C}(K)$ contains three pairs $(P_i, \iota(P_i))$ ($i=1,2,3)$ of torsion points of order $2g+1$.
Without loss of generality  \cite{Lock}  we may assume that there exists a monic separable degree $2g+1$ polynomial $f(x)\in K[x]$
such that $(\mathcal{C},\infty)=(\mathcal{C}_f,\infty)$ and
$$x(P_1)=0,\ x(P_2)=-1,\ x(P_3)=b,$$
where
$$b \in K, \ b \ne 0,-1.$$

It follows from \cite[Th.1]{BekkerZarhin20} that there  exist
polynomials  $v_1(x), v_2(x),v_3(x)\in K[x]$, whose degrees do not exceed   $g$, such that
$v_1(0)\neq0, v_2(b)\neq0, v_3(-1)\neq0$ and
\begin{equation}
\label{v1v2v3}
f(x)=x^{2g+1}+v_1^2(x)=(x-b)^{2g+1}+v_2^2(x)=(x+1)^{2g+1}+v_3^2(x).
\end{equation}
Combining \eqref{v1v2v3}  with Corollary \ref{old}, we obtain that  there exist $g$-element sets $I,J \subset M(2g+1)$ and $\mu,\nu \in K^{*}$ such  that
 $$v_1(x)=\frac{\mu}2\Eta_I(x)+\frac{2g+1}{2\mu}\Eta_{\complement I}(x)=
\frac{\nu}2\Psi_{J,b}(x)-\frac{(2g+1)b}{2\nu}\Psi_{\complement J,b}(x).$$

For each  $g$-element subset $I$ of $M(2g+1)$ and $\mu\in K^{\ast}$ consider the marked hyperelliptic curve $(\mathcal C_{I,\mu},\infty)$, where
$$\mathcal C_{I,\mu}: y^2=x^{2g+1}+ \left(\frac{\mu}2\Eta_I(x)+\frac{2g+1}{2\mu}\Eta_{\complement I}(x)\right)^2.$$
We have just proven that if $\mathcal C_{I,\mu}$ contains three pairs $(P_i, \iota(P_i))$ ($i=1,2,3)$ of torsion points of order $2g+1$, then
there exist a $g$-element subset $J\subset M(2g+1)$ and $\nu\in K^{\ast}$ such that
\begin{equation}\label{mu}\frac{\mu}2\Eta_I(x)+\frac{2g+1}{2\mu}\Eta_{\complement I}(x)=
\frac{\nu}2\Psi_{J,b}(x)-\frac{(2g+1)b}{2\nu}\Psi_{\complement J,b}(x).\end{equation}

\section{Proof of Theorems \ref{finiteness3} and \ref{primeFields}}
\label{sec4}

Let $S(g,K)$ be the set of the marked genus $g$ hyperelliptic curves $(\mathcal C_{I,\mu},\infty)$ over $K$ for each of which  there exist a $g$-element subset $J\subset M(2g+1)$ and $\nu\in K^{\ast}$ such that equality \eqref{mu} holds. It follows from the previous section that every  marked genus $g$ odd degree hyperelliptic curve $(\mathcal{C},\mathcal{O})$ over $K$
such that  $\mathcal{C}(K)$ contains three pairs $(P_i, \iota(P_i))$ ($i=1,2,3)$ of torsion points of order $2g+1$ is   $K$-biregularly isomorphic to a curve from $S(g,K)$. To complete the proof of Theorem \ref{finiteness3} it is sufficient to show that $S(g,K)$ is finite and $\#(S(g,K)) \le 9 (4g-1) \binom{2g}{g}^2$.

Dividing both sides of  equality \eqref{mu} by
 ${\sqrt{2g+1}}/{2}$ and introducing the notation
 $$\frac{\mu}{\sqrt{2g+1}}=\lambda_1,  \
  \frac{\nu}{\sqrt{2g+1}}=\la_2,$$ we get
  \beq\label{main}\lambda_1\Eta_I(x)+\frac{1}{\lambda_1}\Eta_{\complement I}(x)=\lambda_2
\Psi_{J,b}(x)-\frac{b}{\lambda_2}\Psi_{\complement J,b}(x).\eeq

We will prove that the latter equality may hold only for a finite number of
$\lambda_1, \lambda_2$, and $b$. Since there are only finitely many pairs of subsets $I,J$, it is sufficient to prove that for fixed $I,J$ there are finitely many $\lambda_1, \lambda_2,b$ for which \eqref{main} holds.

We divide the proof in two steps. First we prove that for fixed $b$ there are only finitely many pairs $\lambda_1,\lambda_2$
such that equality \eqref{main} is valid.  In addition, we check that if $b\in K_0$, then all such
 $\lambda_1,\lambda_2$ also lie in $K_0$.
Then we prove that  $b$  satisfies an equation whose coefficients lie in $K_0$ and do not depend on
$\lambda_1$ and $\lambda_2$.  More precisely, we  check that each such $b$ is a root of a certain degree $4g-1$ polynomial with coefficients in $K_0$. Since $K_0$ is algebraically closed, we get $b \in K_0$, which proves Theorem \ref{primeFields}.

\subsection*{Step 1}
 We need the following  result.
 \begin{lem}\label{lem}
Let $\alpha,\beta,\gamma,\delta, b\in K$ and $b\neq 0,$ $b\neq-1$.

\begin{itemize}
\item[(i)]
The system of equations
\begin{equation}\label{2}\begin{cases}
\la_1+\frac1{\la_1}=\la_2-\frac b{\la_2},\\
\alpha\la_1+\frac\beta{\la_1}=\gamma\la_2-\frac {\delta}{\la_2}
\end{cases}
\end{equation}
has finitely many solutions in $K$  if and only if not all equalities
 $\alpha=\beta=\gamma=\delta/b$ hold.
 \item[(ii)]
 If   $\alpha,\beta,\gamma,\delta, b\in K_0$ and $b\neq 0,$ $b\neq-1$ and the system \eqref{2}
 has only finitely many solutions in $K$, then all these solutions actually lie in $K_0$.

\item[(iii)]      If the number of solutions of system \eqref{2} is finite, then it does not exceed 9. \end{itemize}
 \end{lem}
  \begin{proof}[Proof of Lemma \ref{lem}]
  Notice that the first equation of the system defines an open dense subset of affine curve
 $\la_1^2\la_2-\la_1\la_2^2+b\la_1+\la_2=0$
 whose projective closure is a nonsingular cubic if
 $b\neq0$, $b\neq-1$. Hence the cubic polynomial $\la_1^2\la_2-\la_1\la_2^2+b\la_1+\la_2$ is irreducible.
 \footnote{Another way to see that this polynomial is irreducible is to observe that the  polynomial
 $\la_1^2\la_2+(b-\la_2^2)\la_1+\la_2$ in $\la_1$ does not vanish for any value of $\la_2\in K$ since $b\neq0$. So, it is irreducible in $K[\la_1,\la_2]$ if and only if it is irreducible in $K(\la_2)[\la_1]$, i.e., if its discriminant $(b-\la_2)^2-4\la_2^2=\la_2^4-(2b+4)\la_2^2+b^2$ is not a square in $K(\la_2)$, which is the case since $(2b+4)^2-4b^2=4b+4\neq0$ due to $b\neq-1$.} The second equation of the system defines an open dense subset of affine cubic curve
   $\alpha\la_1^2\la_2-\gamma\la_1\la_2^2+\delta\la_1+\beta\la_2=0$. This implies that if \eqref{2} has infinitely many solutions, then the (coefficients of the) polynomials
   $\la_1^2\la_2-\la_1\la_2^2+b\la_1+\la_2$ and
$\alpha\la_1^2\la_2-\gamma\la_1\la_2^2+\delta\la_1+\beta\la_2$ should be proportional, which proves (i).
Assertion (ii) follows from the Hilbert's Nullstellensatz applied to the maximal ideals of the ring
$$K_0[\Lambda_1,\Lambda_2]/(\Lambda_1^2\Lambda_2-\Lambda_1\Lambda_2^2+b\Lambda_1+\Lambda_2, \alpha\Lambda_1^2\Lambda_2-\gamma\Lambda_1\Lambda_2^2+\delta\Lambda_1+\beta\Lambda_2)$$
where $\Lambda_1,\Lambda_2$ are independent variables over $K_0$. Assertion (iii) follows from Bezout's theorem
applied to our plane cubics.
\end{proof}

 Let us revisit equation \eqref{main}.  Comparing the leading coefficients of the polynomials on both sides of the equation,
 we get the equality
 $$\la_1+\frac1{\la_1}=\la_2-\frac b{\la_2}.$$ Recall that for each
 $I$ the polynomials $\Eta_I(x)$ and $\Eta_{\complement I}(x)$ that appear in the left-hand side of  \eqref{main} are distinct.  Therefore there exists
  $x_0\in K_0$ such that $\Eta_I(x_0)\neq\Eta_{\complement I}(x_0)$. Applying Lemma  \ref{lem} with
 $\alpha=\Eta_I(x_0)$,    $\beta=\Eta_{\complement I}(x_0)\in K_0$, $\gamma=\Psi_{J,b}(x_0)\in K_0$ and $\delta=b\Psi_{\complement J,b}(x_0)\in K_0$,
 we observe that for each pair
  $I$, $J$ and each $b$ there are only finitely many values
   $\lambda_1,\lambda_2$ such that  \eqref{main} holds. In addition, if $b \in K_0$, then all such  $\lambda_1,\lambda_2$ also lie in $K_0$.
 \subsection*{Step 2}
Let us fix $g$-element subsets $I$ and $J$ of $M(2g+1)$ and prove that
if $b$ satisfies \eqref{main}, then it is a root of a certain nonzero polynomial
 whose coefficients lie in $K_0$ and do not depend on
$\lambda_1$ and $\lambda_2$.


Since $g\geq 2$, each of the sets $I$ and $\complement I$  contains at least 2 elements; let  us choose  two {\sl distinct} elements $\eps_1,\eps_2$ of $I$ and   two {\sl distinct} elements $\eps_3,\eps_4$  of the $\complement I$. Let us put
$$\eta_i: =\eta(\eps_i) \in K_0 \ \forall i=1,2,3,4.$$
Clearly, all $\eta_i$ are distinct and
$$\Eta_I(\eta_1)=\Eta_I(\eta_2)=0,\ \Eta_{\complement I}(\eta_1)\ne 0, \Eta_{\complement I}(\eta_2)\ne 0,$$
$$\Eta_{\complement I}(\eta_3)=\Eta_{\complement I}(\eta_4)= 0,\ \Eta_I(\eta_3)\ne 0, \Eta_I(\eta_4)\ne 0.$$

Successively substituting $x=\eta_1,\eta_2,\eta_3,\eta_4$ into \eqref{main}, we get
\beq\label{eta1}\frac{1}{\lambda_1}\Eta_{\complement I}(\eta_1)=\lambda_2
\Psi_{J,b}(\eta_{1})-\frac{b}{\lambda_2}\Psi_{\complement J,b}(\eta_{1}),\eeq
\beq\label{eta2}\frac{1}{\lambda_1}\Eta_{\complement I}(\eta_2)=\lambda_2
\Psi_{J,b}(\eta_{2})-\frac{b}{\lambda_2}\Psi_{  \complement J,b}(\eta_{2}),\eeq
 \beq\label{eta3}{\lambda_1}\Eta_{I}(\eta_3)=\lambda_2
\Psi_{J,b}(\eta_{3})-\frac{b}{\lambda_2}\Psi_{\complement J,b}(\eta_{3}),\eeq
 \beq\label{eta4}{\lambda_1}\Eta_{ I}(\eta_{4})=\lambda_2
\Psi_{J,b}(\eta_{4})-\frac{b}{\lambda_2}\Psi_{\complement J,b}(\eta_{4}).\eeq
 Consequently,
 \beq\label{eq1}\begin{aligned}
 \Eta_{\complement I}(\eta_1)\Eta_{I}(\eta_3)&=(\lambda_2
\Psi_{J,b}(\eta_{1})-\frac{b}{\lambda_2}\Psi_{\complement J,b}(\eta_{1}))(\lambda_2
\Psi_{J,b}(\eta_{3})\\-\frac{b}{\lambda_2}\Psi_{\complement J,b}(\eta_{3}))\\
&=\lambda_2^2\Psi_{J,b}(\eta_{1})\Psi_{J,b}(\eta_{3})+\frac{b^2}{\lambda_2^2}\Psi_{\complement J,b}(\eta_{1})
\Psi_{\complement J,b}(\eta_{3})\\&-(\Psi_{J,b}(\eta_{1})\Psi_{\complement J,b}(\eta_{3})+\Psi_{\complement J,b}(\eta_{1})
\Psi_{J,b}(\eta_{3}))b,
  \end{aligned}
 \eeq
\beq\label{eq2}\begin{aligned}
  \Eta_{\complement I}(\eta_1)\Eta_{I}(\eta_4)&=(\lambda_2
\Psi_{J,b}(\eta_{1})-\frac{b}{\lambda_2}\Psi_{\complement J,b}(\eta_{1}))(\lambda_2
\Psi_{J,b}(\eta_{4})-\frac{b}{\lambda_2}\Psi_{\complement J,b}(\eta_{4}))\\
&=\lambda_2^2\Psi_{J,b}(\eta_{1})\Psi_{J,b}(\eta_{4})+\frac{b^2}{\lambda_2^2}\Psi_{\complement J,b}(\eta_{1})
\Psi_{\complement J,b}(\eta_{4})\\&-(\Psi_{J,b}(\eta_{1})\Psi_{\complement J,b}(\eta_{4})+\Psi_{\complement J,b}(\eta_{1})
\Psi_{J,b}(\eta_{4}))b,
  \end{aligned}
 \eeq
  \beq\label{eq3}\begin{aligned}
 \Eta_{\complement I}(\eta_2)\Eta_{I}(\eta_3)&=(\lambda_2
\Psi_{J,b}(\eta_{2})-\frac{b}{\lambda_2}\Psi_{\complement J,b}(\eta_{2}))(\lambda_2
\Psi_{J,b}(\eta_{3})-\frac{b}{\lambda_2}\Psi_{\complement J,b}(\eta_{3}))\\
&=\lambda_2^2\Psi_{J,b}(\eta_{2})\Psi_{J,b}(\eta_{3})+\frac{b^2}{\lambda_2^2}\Psi_{\complement J,b}(\eta_{2})
\Psi_{\complement J,b}(\eta_{3})\\&-(\Psi_{J,b}(\eta_{2})\Psi_{\complement J,b}(\eta_{3})+\Psi_{\complement J,b}(\eta_{2})
\Psi_{J,b}(\eta_{3}))b,
  \end{aligned}
 \eeq
\beq\label{eq4}\begin{aligned}
  \Eta_{\complement I}(\eta_2)\Eta_{I}(\eta_4)&=(\lambda_2
\Psi_{J,b}(\eta_{2})-\frac{b}{\lambda_2}\Psi_{\complement J,b}(\eta_{2}))(\lambda_2
\Psi_{J,b}(\eta_{4})-\frac{b}{\lambda_2}\Psi_{\complement J,b}(\eta_{4}))\\
&=\lambda_2^2\Psi_{J,b}(\eta_{2})\Psi_{J,b}(\eta_{4})+\frac{b^2}{\lambda_2^2}\Psi_{\complement J,b}(\eta_{2})
\Psi_{\complement J,b}(\eta_{4})\\&-(\Psi_{J,b}(\eta_{2})\Psi_{\complement J,b}(\eta_{4})+\Psi_{\complement J,b}(\eta_{2})
\Psi_{J,b}(\eta_{4}))b.
  \end{aligned}
 \eeq
 Multiplying \eqref{eq1} by $\Psi_{\complement J,b}(\eta_{4})$, \eqref{eq2} by $\Psi_{\complement J,b}(\eta_{3})$, and
 and subtracting the resulting equations, we get
 \beq\label{eq5}\begin{aligned}
  \Eta_{\complement I}(\eta_1)(\Eta_{I}(\eta_3)\Psi_{\complement J,b}(\eta_4)- \Eta_{I}(\eta_4)\Psi_{\complement J,b}(\eta_3))\\=\lambda_2^2\Psi_{J,b}(\eta_1)(\Psi_{J,b}(\eta_3)\Psi_{\complement J,b}(\eta_4)-
  \Psi_{J,b}(\eta_4)\Psi_{\complement J,b}(\eta_3))\\-\Psi_{\complement J,b}(\eta_1)(\Psi_{J,b}(\eta_4)\Psi_{\complement J,b}(\eta_3)-\Psi_{J,b}(\eta_3)\Psi_{\complement J,b}(\eta_4))b.
   \end{aligned}
   \eeq
   To simplify notation, we put $[g,h]_{u,v}:=g(u)h(v)-g(v)h(u)$ for arbitrary polynomials $g,h\in K[x]$. In this notation \eqref{eq5} takes the form
   \beq\label{eq6}
   \Eta_{\complement I}(\eta_1)[\Eta_I,\Psi_{\complement J,b}]_{\eta_3,\eta_4}=\lambda_2^2\Psi_{J,b}(\eta_1)[\Psi_{J,b},\Psi_{\complement J,b}]_{\eta_3,\eta_4}
   -\Psi_{\complement J,b}(\eta_1)[\Psi_{J,b},\Psi_{\complement J,b}]_{\eta_3,\eta_4}b
   \eeq
  Similarly, from equations \eqref{eq3} and \eqref{eq4}, we get
 \beq\label{eq7}
   \Eta_{\complement I}(\eta_2)[\Eta_I,\Psi_{\complement J,b}]_{\eta_3,\eta_4}=\lambda_2^2\Psi_{J,b}(\eta_2)[\Psi_{J,b},\Psi_{\complement J,b}]_{\eta_3,\eta_4}
   -\Psi_{\complement J,b}(\eta_2)[\Psi_{J,b},\Psi_{\complement J,b}]_{\eta_3,\eta_4}b
   \eeq
   Multiplying \eqref{eq6} by $\Psi_{J,b}(\eta_2)$ ,  \eqref{eq7} by  $\Psi_{J,b}(\eta_1)$, and subtracting obtained equations,
   we get
   \beq \label{main2}
   [\Eta_{\complement I}, \Psi_{J,b}]_{\eta_1,\eta_2}[\Eta_I, \Psi_{\complement J,b}]_{\eta_3,\eta_4}=[\Psi_{J,b},\Psi_{\complement J,b}]_{\eta_1,\eta_2}[\Psi_{J,b},\Psi_{\complement J}]_{\eta_3,\eta_4}b.
         \eeq
         Let us find the leading coefficient of $[\Psi_{J,b},\Psi_{\complement J,b}]_{\eta_1,\eta_2}$ as a polynomial in  $b$.

         \begin{lem}\label{l2} Let $\gamma, \delta,\alpha_i,\beta_j\in K$ and $$r(x)=\prod_{i=1}^g(x+b\alpha_i), s(x)=\prod_{j=1}^g(x+b\beta_j)\in K[b][x].$$ Then the coefficient of $[r,s]_{\gamma,\delta}\in K[b]$ at $b^k$ is zero for $k\geq 2g$ and the coefficient at $b^{2g-1}$
         is equal to $$\alpha_1\ldots\alpha_g\beta_1\ldots\beta_g(\gamma-\delta)\left((1/{\alpha_1}+\cdots+1/{\alpha_g})
         -(1/\beta_1+\cdots+1/\beta_g)\right).$$
         \end{lem}
         \begin{proof} We have
         $$[r,s]_{\gamma,\delta}=\prod_{i=1}^g(\gamma+b\alpha_i)\prod_{j=1}^g(\delta+b\beta_j)-
         \prod_{j=1}^g(\delta+b\alpha_j)\prod_{i=1}^g(\gamma+b\beta_i).$$
          Obviously,  the coefficients of $[r,s]_{\gamma,\delta}$ at  $b^k$ are zero for all $k>2g$. The coefficient at $b^{2g}$ is also zero. The coefficient of $[r,s]_{\gamma,\delta}$ at $b^{2g-1}$ is
          $$\alpha_1\ldots\alpha_g\beta_1\ldots\beta_g(\gamma-\delta)\left((1/{\alpha_1}+\cdots+1/{\alpha_g})
         -(1/\beta_1+\cdots+1/\beta_g)\right).$$
         \end{proof}
         Now we apply Lemma~\ref{l2} to $[\Psi_{J,b},\Psi_{\complement J,b}]_{\eta_1,\eta_2}$ and $ [\Psi_{J,b},\Psi_{\complement J,b}]_{\eta_3,\eta_4}$.
         At this point we need to use the conditions (a)-(b) of our Theorems imposed on $\fchar(K)$.

         \begin{lem}\label{main3}
         The polynomials     $ [\Psi_{J,b},\Psi_{\complement J,b}]_{\eta_1,\eta_2}$ and $ [\Psi_{J,b},\Psi_{\complement J,b}]_{\eta_3,\eta_4}$ have degree $2g-1$  with respect to $b$.
         \end{lem}

\begin{proof}
         Let us consider the polynomial     $ [\Psi_{J,b},\Psi_{\complement J,b}]_{\eta_1,\eta_2}$ and prove that its degree with respect to  $b$ is $2g-1$.

 We want to apply Lemma \ref{l2} to $r(x)=\Psi_{J,b}(x), s(x)=\Psi_{\complement J,b}$. In the notation of  Lemma \ref{l2}, we have
 $$\{\alpha_1, \dots \alpha_g\}=\eta(J):=\{\eta(\eps)\mid \eps\in J\},$$
 $$\{\beta_1, \dots \beta_g\}=\eta(\complement J):=\{\eta(\eps)\mid \eps\in \complement J\}.$$
          Let $J=\{\alpha_1,\ldots,\alpha_g\}$ and $\complement J=\{\beta_1,\ldots, \beta_g\}$ (of course,
          $J\cup \complement J=\{\eta_1,\ldots,\eta_{2g}\}$; we introduced $\alpha_i$ and $\beta_j$ for convenience).
          Let $\eta_1=\alpha_i$ and $\eta_2=\alpha_j$, $i\neq j$.

          Let us assume that the degree of $ [\Psi_{J,b},\Psi_{\complement J,b}]_{\eta_1,\eta_2}$ with respect to $b$ is less than $2g-1$.
          Then the coefficient  of  $ [\Psi_{J,b},\Psi_{\complement J,b}]_{\eta_1,\eta_2}$ at $b^{2g-1}$ is zero. By Lemma~\ref{l2} applied to $\gamma=\eta_1$ and $\delta=\eta_2$,
          $$\sum_{\eps \in J} \frac{1}{\eta(\eps)}=\sum_{\eps \in\complement  J} \frac{1}{\eta(\eps)}.$$
         Since
       $$\begin{aligned} \sum_{\eps \in J} \frac{1}{\eta(\eps)}+\sum_{\eps \in\complement  J} \frac{1}{\eta(\eps)}=
       \sum_{\eps\in M(2g+1)}1/\eta(\eps)\\=\sum_{\eps\in M(2g+1)}(\eps-1)=
       -2g+\sum_{\eps\in M(2g+1)}\eps=-2g-1,\end{aligned}$$
         we have
         $$\sum_{\eps \in J} \frac{1}{\eta(\eps)}=-\frac{2g+1}{2},$$
         which gives the required contradiction  if $\fchar(K)=0$, because
         $$ \sum_{\eps \in J} \frac{1}{\eta(\eps)}= \sum_{\eps \in J} (\eps-1)$$
         is an algebraic integer while  $\frac{2g+1}{2}$ is {\sl not} an algebraic integer.
         The obtained contradiction implies that $p:=\fchar(K)>0$ and  therefore
         $\frac{2g+1}{2}\in \mF_p$.

         Let
         $$\gamma \in M(2g+1)\in K_0^{*}\subset K^{*}$$
         be a primitive $(2g+1)$th root of unity in $K$.  By our assumptions,  the minimal polynomial of $\gamma$ over $\mF_p$ has degree $2g$.

         Let us consider the bijective discrete logarithm map
         $$\log_{\gamma}: M(2g+1) \cong \{1,2, \dots,  2g-1,2g\}$$
         defined by
         $$\eps=\gamma^{\log_{\gamma}(\eps)} \ \forall \eps\in M(2g+1).$$
         For example,
         $$\log_{\gamma}(\gamma)=1, \ \log_{\gamma}(1/\gamma)=2g.$$
         Then $\gamma$ is a common root of two nonconstant polynomials
         $$s_J(t)=\frac{1}{2}+ \sum_{\eps \in J} t^{\log_{\gamma}(\eps)} \in \mF_p[t], \
         s_{\complement J}(t)=\frac{1}{2}+ \sum_{\eps \in \complement J} t^{\log_{\gamma}(\eps)} \in \mF_p[t].$$
        Clearly, the degrees of both $s_J(t)$ and $s_{\complement J}(t)$ do not exceed $2g$.
        On the other hand, since both $s_J(t)$ and $s_{\complement J}(t)$ are divisible by the minimal degree $2g$ polynomial of $\gamma$,
         $$\deg(s_J(t))=2g,  \ \deg(s_{\complement J}(t))=2g.$$
         If $J$  does not contain $1/\gamma$, then the degree of $\deg(s_J(t))<2g$, which is wrong.  Hence $J$  does  contain $1/\gamma$, which implies that  $\deg(s_{\complement J}(t))<2g$, which is also wrong.  The obtained contradiction implies that  the polynomial     $ [\Psi_{J,b},\Psi_{\complement J,b}]_{\eta_1,\eta_2}$  has degree $2g-1$ in $b$.

Replacing $J$ by $\complement J$ and $\eta_1,\eta_2$ by $\eta_3$, $\eta_4$, we obtain that the degree of $[\Psi_{J,b},\Psi_{\complement J,b}]_{\eta_3,\eta_4}$ in $b$ is also $2g-1$.

         \end{proof}

           \begin{prop}\label{prop2}
         The  polynomial
         $$ [\Psi_{J,b},\Psi_{\complement J,b}]_{\eta_1,\eta_2}[\Psi_{J,b},\Psi_{\complement J,b}]_{\eta_3,\eta_4}b- [\Eta_{\complement I}, \Psi_{J,b}]_{\eta_1,\eta_2}[\Eta_I, \Psi_{\complement J,b}]_{\eta_3,\eta_4} \in K_0[b]$$
         has degree $4g-1$  with respect to $b$ (and, in particular, is nonzero).
         \end{prop}

         \begin{proof} By Lemma \ref{main3} the degree of the polynomial $$ [\Psi_{J,b},\Psi_{\complement J,b}]_{\eta_1,\eta_2}[\Psi_{J,b},\Psi_{\complement J,b}]_{\eta_3,\eta_4}$$ with respect to $b$ is $4g-2$.
         Since the degree of  $$ [\Eta_{\complement I}, \Psi_{J,b}]_{\eta_1,\eta_2}[\Eta_I, \Psi_{\complement J,b}]_{\eta_3,\eta_4}$$ with respect to $b$ does not exceed $2g$, the degree of $$ [\Psi_{J,b},\Psi_{\complement J,b}]_{\eta_1,\eta_2}[\Psi_{J,b},\Psi_{\complement J,b}]_{\eta_3,\eta_4}b- [\Eta_{\complement I}, \Psi_{J,b}]_{\eta_1,\eta_2}[\Eta_I, \Psi_{\complement J,b}]_{\eta_3,\eta_4}$$ is $4g-1$.
     \end{proof}
Let us complete the proof of Theorems \ref{finiteness3} and \ref{primeFields}. We proved that the set $S(g,K)$ is finite. There are exactly $\binom{2g}{g}^2$ pairs
of subsets $I,J\subset M(2g+1)$. For each such pair, $b$ satisfies equation \eqref{main2},  which has degree $4g-1$ by Proposition \ref{prop2}.
It follows from assertion (iii) of Lemma~\ref{lem} that  for each triple  $I,J,b$ we have at most 9 possible pairs $\lambda_1,\lambda_2$ such that the triple $\lambda_1,\lambda_2, b$ satisfies \eqref{main}.
 This implies the
desired
estimate
$$\#(S(g,K)) \le 9 (4g-1) \binom{2g}{g}^2,$$
 which ends the proof of Theorem \ref{finiteness3}.
 Theorem~\ref{primeFields} follows from assertion (ii) of Lemma \ref{lem}.
\begin{Rem}
As was pointed out by the referee, the bound in Theorem \ref{finiteness3} may be improved slightly, using the following observation.
Let $\mathcal C_b$ be a curve having three pairs of   order $2g+1$ points whose $x$-coordinates are $0,-1,b$.
Then $C_b$ is isomorphic to $\mathcal C_{b^{\prime}}$, where
$$b^{\prime}\in\{1/b,-(b+1),-1/(b+1),-(b+1)/b,-b/(b+1)\}.$$
\end{Rem}
{\bf Acknowledgements}. The authors are grateful to Oksana Podkopaeva for reading a preliminary version of the manuscript and making valuable comments. Our special thanks go the referee for careful reading of the manuscript and helpful comments.

 \end{document}